\documentclass[10pt,twoside]{article}
\usepackage{amsfonts,epsfig}
\usepackage{amsmath,amssymb, amsthm}
\usepackage{enumerate}
\usepackage{hyperref}
\usepackage{tikz}

\newtheorem{thm}{Theorem}[section]
\newtheorem{lem}[thm]{Lemma}
\newtheorem{obs}[thm]{Observation}
\newtheorem{fact}[thm]{Fact}
\newtheorem{conj}[thm]{Conjecture}
\newtheorem{ex}[thm]{Example}
\newtheorem{rem}[thm]{Remark}

\newcommand{\lc}{\left\lceil}
\newcommand{\rc}{\right\rceil}
\newcommand{\lf}{\left\lfloor}
\newcommand{\rf}{\right\rfloor}

\def\tr{\operatorname{tr}}

\definecolor{red}{rgb}{.8,0,0}

\definecolor{blu}{rgb}{0,0,1}

\newcommand{\D}{\mathcal D}

\newcommand{\bit}{\begin{itemize}}
\newcommand{\eit}{\end{itemize}}
\newcommand{\ben}{\begin{enumerate}}
\newcommand{\een}{\end{enumerate}}
\newcommand{\beq}{\begin{equation}}
\newcommand{\eeq}{\end{equation}}
\newcommand{\bea}{\begin{eqnarray*}}
\newcommand{\eea}{\end{eqnarray*}}
\newcommand{\bpf}{\begin{proof}}
\newcommand{\epf}{\end{proof}}

\begin{document}



\bibliographystyle{plain}

\title{Proof of a conjecture of Graham and Lov\'asz concerning unimodality of coefficients of the distance characteristic polynomial of a tree\thanks{Received by the editors on Month x, 200x.
Accepted for publication on Month y, 200y   Handling Editor: .}}

\author{Ghodratollah Aalipour\thanks{Department of Mathematics and Computer Sciences, Kharazmi University, 50 Taleghani St., Tehran, Iran and Department of Mathematical and Statistical Sciences, University of Colorado Denver, CO, USA (alipour.ghodratollah@gmail.com).}\and 
Aida Abiad\thanks{Department of Econometrics and Operations Research, Tilburg University, Tilburg, The Netherlands, (aidaabiad@gmail.com).} \and
 Zhanar Berikkyzy\thanks{Department of Mathematics, Iowa State University, Ames, IA 50011, USA (zhanarb@iastate.edu, chlin@iastate.edu).}\and  
 Leslie Hogben\thanks{Department of Mathematics, Iowa State University, Ames, IA 50011, USA (hogben@iastate.edu) and American Institute of Mathematics, 600 E. Brokaw Road, San Jose, CA 95112, USA (hogben@aimath.org).} \and
 Franklin H.J. Kenter \thanks{Mathematics Department, United State Naval Academy, Chauvenet Hall, 572C Holloway Road, Annapolis, MD 21402, USA (franklin.kenter@gmail.com).}\and 
 Jephian C.-H. Lin\footnotemark[4]\and   
 Michael Tait\thanks{Department of Mathematics, University of California San Diego, La Jolla CA, 92037, USA (mtait@cmu.edu).}
} 
\date{}

\maketitle

\begin{abstract} 
{We establish a conjecture of Graham and Lov\'asz  that the (normalized) coefficients of the distance characteristic polynomial of a tree are unimodal and prove they are log-concave.  We also establish upper and lower bounds on the location of the peak.


}
\end{abstract}

\section{Introduction}\label{sintro}

The \emph{distance matrix} $\D(G)$ of a simple, finite, undirected, connected graph $G$ is the matrix indexed by the vertices of $G$ with $(i,j)$-entry equal to the distance  between the vertices $v_i$ and $v_j$, i.e., the length of a shortest path between $v_i$ and $v_j$. The characteristic polynomial of $\D(G)$ is defined by $p_{\D(G)}(x)=\det(xI-\D(G))$ and is called the \emph{distance characteristic polynomial} of $G$. Since $\D(G)$ is a real symmetric matrix, all of  the roots of the distance characteristic polynomial are real.  Distance matrices were introduced in the study of a data communication problem in \cite{GP71}.  This problem involves finding appropriate addresses so that a message can move efficiently through a series of loops from its origin to its destination, choosing the best route at each switching point.  Recently there has been renewed interest in the loop switching problem \cite{PersiTalk}.  There has also been extensive work on distance spectra; see \cite{AH14} for a recent survey.

A sequence $a_0,a_1,a_2,\ldots, a_n$ of real numbers is {\em unimodal} if there is a $k$ such that $a_{i-1}\leq a_{i}$ for $i\leq k$ and $a_{i}\geq a_{i+1}$ for $i\geq k$, and the sequence is {\em log-concave} if $a_j^2\geq a_{j-1}a_{j+1}$ for all $j=1,\dots,  n-1$.
 Recent surveys about unimodality and related topics can be found in \cite{brandensurvey, B15}, and a classical presentation is given in \cite{comtet}.
 
 For a graph $G$ on $n$ vertices, the coefficient in $\det(\D(G)-xI)=(-1)^np_{\D(G)}(x)$ of $x^k$  is denoted by $\delta_k(G)$ by Graham and Lov\'asz \cite{GL78}, so the coefficient of $x^k$ in $p_{\D(G)}(x)$ is $(-1)^n\delta_k(G)$.  
  The following statement appears on page~83 in \cite{GL78} (a {\em tree} is a connected graph that does not have cycles, and $n$ is its {\em order}, i.e., number of vertices):

\begin{quote}{\em  It appears that in fact for each tree $T$, the quantities \break$(-1)^{n-1}\delta_k(T)/2^{n-k-2}$ are {unimodal}  with the maximum value occurring for $k=\big\lfloor \frac n 2\big\rfloor$.  We see no way to prove this, however.
}\end{quote}

 \begin{fact}\label{GLsign}{\rm \cite[Equation (44)]{GL78}} For a tree $T$ on $n$ vertices, \[ (-1)^{n-1}\delta_k(T) > 0\qquad\mbox{ for }0\le k\le n-2.\vspace{-10pt}\]   \end{fact}

Throughout this discussion, the order of a graph is assumed to be at least three (any sequence $a_0$ is trivially unimodal and the peak location is 0). For a graph $G$ of order $n$ and $0\le k\le n-2$, define $ d_k(G)=\left(\frac 1 {2^{n-2}}\right)2^k |\delta_k(G)|$. We call  the numbers $d_k(G)$ the {\em normalized coefficients}.  If $T$ is a tree, then $d_k(T)=(-1)^{n-1}\delta_k(T)/2^{n-k-2}$ by Fact \ref{GLsign}.  For a tree, the normalized coefficients represent counts of certain subforests of the tree \cite{GL78}.   The conjecture in \cite{GL78} can be rephrased as: 

\begin{quote}{\em 
For a tree $T$ of order $n$, the sequence of normalized coefficients \break $d_0(T),\dots,d_{n-2}(T)$ is unimodal and  the peak occurs at $\big\lfloor \frac n 2\big\rfloor$.}\end{quote}

The conjecture regarding the location of the peak was disproved by Collins \cite{C89} who showed that for both stars and paths the sequence $d_0(T),\dots,d_{n-2}(T)$ is unimodal, but for paths the peak is at approximately $\big(1-\frac 1 {\sqrt 5}\big)n$ (and at $\big\lfloor \frac n 2\big\rfloor$ for stars).\footnote{Despite use of the term {\em coefficient}  throughout \cite{C89}, the sequence discussed there  is $d_k(T)$, not $\delta_k(T)$.}  Conjecture~9 in \cite{C89}, which Collins attributes to Peter Shor, is: 

\begin{conj}[Collins-Shor]The [normalized] coefficients  of the distance characteristic polynomial for any tree $T$ with $n$ vertices are unimodal  with peak between $\big\lfloor \frac n 2\big\rfloor$ and $\big\lceil n- \frac n {\sqrt 5}\big\rceil$.\label{peakconj}
\end{conj}

In  \cite{C89}, Conjecture 9 is stated without the floor or ceiling;  $\big\lfloor \frac n 2\big\rfloor$ is clearly  the intended  lower bound, since \cite[Theorem 1]{C89} establishes $\lf \frac n 2\rf$  as the peak location for a star.  An examination of the proof of   \cite[Theorem~3]{C89} shows that the ceiling is needed in the upper bound (although the path $P_n$ may attain either the floor or the ceiling depending on $n$).    
This conjecture is included in \cite{AH14} as Conjecture~2.6 (again without ``normalized" and without the floor and ceiling), followed by the comment, {\em ``No more results are known about that conjecture."} 

The log-concavity of the sequences $d_k(T)$ of normalized coefficients and $|\delta_k(T)|$  of absolute values of coefficients   are equivalent,  and we show  in Theorem~\ref{Tunimod} below that both sequences $|\delta_0(T)|,\dots,|\delta_{n-2}(T)|$ and   $d_0(T),\dots,d_{n-2}(T)$ are log-concave and unimodal.  In Section~\ref{sspeak} we  establish an upper bound of $\lc\frac 2 3 n\rc$ for the peak location of the normalized coefficients. We also show that the coefficient $\frac 2 3$ can be improved when the tree is ``star-like'' with many paths of length 2.  Further, we give a lower bound of $\frac{n}{d+1}$ where $d$ is the diameter of the tree (i.e.,  the number of edges in a longest path in the tree). Finally,  in Section~\ref{ssgraph} we give an example showing unimodality need  not be  true for graphs that are not trees.

To establish these results, we need some additional definitions and facts. The next observation is immediate from the definition.

\begin{obs}\label{modfromplain} 
Let $a_0,a_1,a_2,\ldots, a_n$ be a sequence of real numbers, let $c$ and $s$ be nonzero real numbers, and define $b_k=sc^ka_k$.  Then $a_0,a_1,a_2,\ldots, a_n$ is log-concave if and only if $b_0,b_1,b_2,\ldots, b_n$ is log-concave. \end{obs}

   Consider a real polynomial $p(x)=a_nx^n+\dots + a_1x+a_0$.  The \emph{coefficient sequence} of $p$ is the sequence $a_0,a_1,a_2,\ldots, a_n$. The polynomial $p$ is \emph{real-rooted} if all roots of $p$ are real (by convention, constant polynomials are considered real-rooted).   The next result is known (see, for example, \cite{brandensurvey, B15, comtet}).  It is straightforward to adapt the proof  of \cite[Lemma 1.1]{brandensurvey} 
 or  \cite[Theorem B, p.~270]{comtet}, which  are stated with the additional assumption that the polynomial coefficients are nonnegative, to the more general case.

\begin{lem}\label{unimodequiv}$\null$
\begin{enumerate}[(a)]
\item\label{rrlc} If $p(x)=a_nx^n+\dots + a_1x+a_0$ is a real-rooted polynomial, then:
\begin{enumerate}[(i)]
\item\label{rrlc1} $\frac{a_j^2}{\binom{n}{j}^2} \geq \frac{ a_{j+1}{a_{j-1}}}{\binom{n}{j+1}\binom{n}{j-1}}$ for $j=1,\dots,n-1$.
\item\label{rrlc2} The coefficient sequence $a_i$ of $p$ is log-concave.
\end{enumerate}
\item\label{lcum} If $a_0,a_1,a_2,\dots, a_n$ is positive and log-concave, then $a_0,a_1,a_2,\ldots, a_n$ is unimodal.
\end{enumerate}
\end{lem}


\section{Proof of Graham and Lov\'asz' unimodality conjecture for the distance characteristic polynomial of a tree}\label{sunimod}$\null$

\begin{thm}\label{Tunimod} 
Let $T$ be  a tree  of order $n$.
\ben
\item[$(i)$] The  coefficient sequence  of the distance characteristic polynomial $p_{\D(T)}(x)$ is log-concave.
\item[$(ii)$] The sequence $|\delta_0(T)|,\dots,|\delta_{n-2}(T)|$ of absolute values  of coefficients of the distance characteristic polynomial  is log-concave and unimodal.
\item[$(iii)$] The sequence $d_0(T),\dots,d_{n-2}(T)$ of normalized coefficients  of the distance characteristic polynomial    is log-concave and unimodal.
\een
\end{thm}
\begin{proof}
Let $\D(T)$ be the distance matrix of $T$. Since $p_{\D(T)}(x)$ is real-rooted, the coefficient sequence
$
(-1)^{n}\delta_0(T),\dots,(-1)^{n}\delta_{n-2}(T), 0,1
$ 
is log-concave by Lemma~\ref{unimodequiv}(i).  

Therefore, the sequence $(-1)^{n}\delta_0(T),\dots,(-1)^{n}\delta_{n-2}(T)$ is log-concave. By Fact \ref{GLsign}, $ (-1)^{n-1}\delta_k(T) > 0$ for $0\le k\le n-2$, so we have $(-1)^n\delta_k(T)<0$ for $0\le k\le n-2$. Since all of the terms $(-1)^{n}\delta_0(T),\ldots,(-1)^{n}\delta_{n-2}(T)$ are negative,  the sequence of their absolute values $\{|\delta_k(T)|\}_{k=0}^{n-2}$ is log-concave and positive. Then by Lemma~\ref{unimodequiv}\eqref{lcum}, the sequence $|\delta_0(T)|,\ldots,|\delta_{n-2}(T)|$ is unimodal.

Since $ d_k(T)=\left(\frac 1 {2^{n-2}}\right)2^k |\delta_k(T)|$, the log-concavity of the sequence $\{d_k(T)\}_{k=0}^{n-2}$ then follows from Observation~\ref{modfromplain}.  Since  $\{d_k(T)\}_{k=0}^{n-2}$ is positive, it is unimodal by Lemma~\ref{unimodequiv}\eqref{lcum}.
\end{proof}


\section{Bounds on the peak location}\label{sspeak}

For a tree $T$ of order $n$, the question of the  location of the peak of the unimodal sequence of normalized coefficients $\{d_k(T)\}_{k=0}^{n-2}$  remains open.  Note that Conjecture \ref{peakconj} says that the peak location is between $\lfloor 0.5n\rfloor$ and roughly $\lceil 0.5528 n\rceil$.  
Computations on {\em Sage} \cite{code, sage} confirm this conjecture for all trees of order at most $20$.  
In this section we show that the peak location is at most $\lceil 0.6667n\rceil$  for all trees of order $n$, and at least $\left \lfloor \frac{n-2}{1+d} \right \rfloor $ for a tree of  diameter $d$ and order $n$. Furthermore,  the upper bound we establish is better for a  ``star-like'' tree, that is, when the tree has a high fraction of the number of paths of length 2 in a star (which attains the maximum possible number of paths of length 2).

\begin{obs}
\label{scaledpoly}
Let $T$ be a tree on $n$ vertices and define
\[\ell_T(x)=-\frac{1}{2^{n-2}}\det(2xI-\D(T)).\]
Then $\ell_T(x)$ is a real-rooted polynomial with coefficients $-4$ for $x^n$, $0$ for $x^{n-1}$, and $d_k(T)>0$ for $x^k$ when $0\leq k\leq n-2$.
\end{obs}

\begin{lem}
\label{peaklem}
Let $a_0,a_1,a_2,\ldots , a_{n-2}$ be a  unimodal sequence with $a_i> 0$ for $i=0,\ldots ,n-2$ such that $\sum_{k=0}^na_kx^k$ is a real-rooted polynomial.  
\begin{enumerate}
\item If for some index $j\neq n,n-1$ 
\[\frac{n-j}{n(j+1)}\cdot \frac{a_1}{a_0}<1,\]  
then the peak location is at most $j$. 
\item If for some index $j\neq n,n-1,0$ 
\[\frac{(n-2)(n-j+1)}{3j}\cdot \frac{a_{n-2}}{a_{n-3}}>1,\]
then the peak location is at least $j$. 
\end{enumerate}
\end{lem}
\begin{proof}
By Lemma \ref{unimodequiv}\eqref{rrlc}(i)
\[a_j^2\geq \frac{{n\choose j}^2}{{n\choose j+1}{n\choose j-1}}a_{j+1}a_{j-1}=\frac{(j+1)(n-j+1)}{j(n-j)}a_{j+1}a_{j-1}.\]
Then
\bea
\frac{a_{j+1}}{a_j} & \leq & \frac{j(n-j)}{(j+1)(n-j+1)}\cdot \frac{a_j}{a_{j-1}} \\ 
 & \leq & \left(\frac{j}{j+1}\cdot \frac{j-1}{j}\cdot\cdots \cdot\frac{1}{2}\right)\left(\frac{n-j}{n-j+1}\cdot\frac{n-j+1}{n-j+2}\cdot\cdots\cdot\frac{n-1}{n}\right)\frac{a_1}{a_0} \\
 & = & \frac{n-j}{n(j+1)}\cdot\frac{a_1}{a_0}.
\eea
If this value is smaller than $1$, then $a_{j+1}<a_j$ and the peak location is at most $j$.

Similarly, 
\bea
\frac{a_{j}}{a_{j-1}} & \geq & \frac{(j+1)(n-j+1)}{j(n-j)}\cdot \frac{a_{j+1}}{a_{j}} \\ 
 & \geq & \left(\frac{j+1}{j}\cdot \frac{j+2}{j+1}\cdot\cdots \cdot\frac{n-2}{n-3}\right)\left(\frac{n-j+1}{n-j}\cdot\frac{n-j}{n-j-1}\cdot\cdots\cdot\frac{4}{3}\right)\frac{a_{n-2}}{a_{n-3}} \\
 & = & \frac{(n-2)(n-j+1)}{3j}\cdot\frac{a_{n-2}}{a_{n-3}}.
\eea
If this value is greater than $1$, then $a_j>a_{j-1}$ and the peak location is at least $j$.
\end{proof}

\begin{thm}
Suppose $T$ is a tree on $n \ge 3$ vertices with at least  $\rho {n-1 \choose 2}$ paths of length $2$ for some  nonnegative real number $\rho$. Then the peak location of the normalized coefficients $d_0(T),d_1(T),\ldots ,d_{n-2}(T)$ is at most $\lc\frac{2-\rho}{3-\rho}n\rc$.  Since $\rho= 0$ applies to every tree, the peak location is at most $\lc\frac{2}{3}n\rc$ for every tree on $n$ vertices.
\end{thm}

\begin{proof}
By Observation \ref{scaledpoly}, we may apply Lemma \ref{peaklem} to $\ell_T(x)$.  When $0\leq j\leq n-2$ and 
\[\frac{n-j}{n(j+1)}\cdot \frac{d_1(T)}{d_0(T)}<1,\]
the peak location is at most $j$.  Since $d_0(T)$ and $d_1(T)$ are both positive numbers, the inequality is equivalent to 
\[j>\frac{rn-n}{n+r}=n-\frac{n^2+n}{n+r},\text{ where }r=\frac{d_1(T)}{d_0(T)}.\]

The formula $d_0(T)=n-1$ is given in  \cite[Theorem 3]{GP71}.  Defining  $N_{P_{3}}(T)$ to be the number of subtrees of $T$ that are isomorphic to the path $P_3$ on three vertices (of length 2), the formula $d_1(T)=2n(n-1)-2N_{P_3}(T)-4$ follows from    \cite[Theorem 4.1]{EGG76}\footnote{Our notation is slightly different but examination of  \cite[Table 2]{EGG76} clarifies the notation.} 
 by using the definition $d_k(T)=(-1)^{n-1}\delta_k(T)/2^{n-k-2}$. Since $\frac 1 2 \rho(n-1)(n-2)=N_{P_3}(T)\ge n-2$,
\[r=\frac{2n(n-1)-2N_{P_3}(T)-4}{n-1} = \frac{2n(n-1)-\rho(n-1)(n-2)-4}{n-1}<(2-\rho)n+2\rho. \]
Now
\[ n-\frac{n^2+n}{n+r} < n-\frac{n^2+n}{(3-\rho)n+2\rho}= n-\frac{n+1}{3-\rho+(2\rho/n)}\leq n-\frac{n}{3-\rho}=\frac{2-\rho}{3-\rho}n. \]
The last inequality follows from  $\frac{n}{3-\rho}\leq \frac{1}{(2\rho/n)}$, which is justified by  $\rho\leq 1$.

Therefore, $j=\lc\frac{2-\rho}{3-\rho}n\rc$ is an upper bound of the peak location.
\end{proof}

\begin{rem}{\rm If  the number $N_{P_3}(T)$ of paths of length two  is known  for every tree $T$ in a particular family, then $\rho$ can be set equal to $\frac{N_{P_3}(T)}{{n-1\choose 2}}$.  For example, for the star $S_n$ on $n$ vertices, $N_{P_3}(S_n)={n-1\choose 2}$, so  $\rho=1$ and $\lc\frac{2-\rho}{3-\rho}n\rc= \lc\frac{n}{2}\rc$.  Thus for a star our upper bound is equal to (if $n$ is even) or one more than (if $n$ is odd) the known value $\lf\frac n 2\rf$ for the peak of the normalized coefficients for $S_n$ \cite[Theorem 1]{C89}.}
\end{rem}

We will utilize a technique similar to the upper bound in order to derive a lower bound. However, we need the following lemma to provide an estimate for the necessary ratio.

\begin{lem}\label{lowrbound}
For any tree $T$ on $n$ vertices with diameter $d$
\[\frac{d_{n-3}(T)}{ d_{n-2}(T)} < \frac 1 3n d. \]
\end{lem}

\begin{proof}
Let $\mathcal{D} := \mathcal{D}(T)$ denote the distance matrix of $T$, and let $\D_{ij}$ denote its $ij$-entry. From \cite[Equations (4c) and (4d)]{EGG76}, 

\[ \delta_{n-2}(T)  = (-1)^{n-1} \sum_{i < j } \mathcal{D}_{ij}^2  \]
\[ \delta_{n-3}(T) = \displaystyle(-1)^{n-1}  \sum_{i < j < k} 2 \mathcal{D}_{ij} \mathcal{D}_{jk} \mathcal{D}_{ki}. \]

We will now express the corresponding normalized coefficients in terms of the traces of powers of $\mathcal{D}$. First, let us consider $d_{n-2}(T)$. Since the diagonal entries of $\D$ are all zero,

\[ d_{n-2}(T)  = \sum_{i < j} \mathcal{D}_{ij}^2 = \frac{1}{2}  \sum_{i}  \sum_{j} \mathcal{D}_{ij} \mathcal{D}_{ji} = \frac{1}{2} \sum_{i}  (\mathcal{D}^2)_{ii}  = \frac{1}{2} \tr(\mathcal{D}^2)  \]
where the second equality follows from $\mathcal{D}$ being symmetric. Similarly, for $d_3(T)$, 
\begin{eqnarray*}
 d_{n-3}(T) &=&  \sum_{i < j < k}  \mathcal{D}_{ij} \mathcal{D}_{jk} \mathcal{D}_{ki} \\
 &=&  \frac{1}{6} \sum_{\substack{i,j,k \\ {\rm different} }}  \mathcal{D}_{ij} \mathcal{D}_{jk} \mathcal{D}_{ki} \\
  &=&  \frac{1}{6} \sum_{i,j,k}  \mathcal{D}_{ij} \mathcal{D}_{jk} \mathcal{D}_{ki} \\
  &=&  \frac{1}{6} \sum_{i} \sum_{j,k}  \mathcal{D}_{ij} \mathcal{D}_{jk} \mathcal{D}_{ki}  = \frac{1}{6} \sum_{i}   (\mathcal{D}^3)_{ii} = \frac{1}{6} \tr(\mathcal{D}^3) \\
  \end{eqnarray*}
where the third line follows because if any two of $i, j, k$ are equal, then the corresponding entry in $\mathcal{D}$ is 0.

Let $\lambda_1 \le  \lambda_2 \le \ldots \le \lambda_n =: \lambda_{max}$ denote the eigenvalues of $\mathcal{D}(T)$.
Since $\tr(\mathcal{D}^2) = \displaystyle \sum_{i}\lambda_i^2$ and similarly $\tr(\mathcal{D}^3) = \displaystyle \sum_{i}\lambda_i^3$, we have

\[ \frac{d_{n-3}(T)}{ d_{n-2}(T)} =\frac {\frac 16} {\frac 1 2}  \frac{\tr(\D(T)^3)}{\tr(\D(T)^2)} = \frac{1}{3}  \frac{ \sum_{i}\lambda_i^3}{ \sum_{i}\lambda_i^2} \le \frac{1}{3}  \frac{ \lambda_{max} \sum_{i}\lambda_i^2}{ \sum_{i}\lambda_i^2}  = \frac 13 \lambda_{max} <  \frac13 nd  \]  where the last inequality comes from that the the row sums of $\D$ are bounded above by $nd$.
\end{proof}

\begin{thm}
Let $T$ be a tree on $n \ge 3$ vertices with diameter $d$. Then, the peak location of the normalized coefficients $d_0(T),d_1(T),\ldots ,d_{n-2}(T)$ is at least $\left\lfloor\frac{n-2}{1+d}\right\rfloor$.
\end{thm}

\begin{proof}
By Observation \ref{scaledpoly}, we may apply Lemma \ref{peaklem} to $\ell_T(x)$.  When $1 \le j\leq n-2$ and 
\[\frac{(n-2)(n-j+1)}{3j}\cdot \frac{d_{n-2}(T)}{d_{n-3}(T)}>1,\]
the peak location is at least $j$.  Since $d_{n-2}(T)$ and $d_{n-3}(T)$ are both positive numbers, the inequality is equivalent to 
\[j<\frac{(n-2)(n+1)}{(n-2)+(3/r)},\text{ where }r=\frac{d_{n-2}(T)}{d_{n-3}(T)}.\]

By applying Lemma \ref{lowrbound}, $\frac{3}{r}<nd$.

Thus,
\bea
\frac{(n-2)(n+1)}{(n-2)+(3/r)} & > &\frac{(n-2)(n+1)}{(1+d)n-2}\\
& = & \frac{n-2}{1+d}\cdot \frac{n+1}{n-2/(1+d)}\\
& > & \frac{n-2}{1+d}.
\eea
So $j=\left\lfloor\frac{n-2}{1+d}\right\rfloor$ is a lower bound of the peak location.
\end{proof}


\section{Graphs that are not trees}\label{ssgraph}
Since the distance matrix of any graph $G$ is a real symmetric matrix, the coefficient sequence of the distance characteristic polynomial of $G$ is log-concave.  However, it need not be the case that all coefficients  
of the distance characteristic polynomial have the same sign.  Thus statements analogous to those in Theorem~\ref{Tunimod} can be false for graphs that are not trees. 

\begin{ex}\label{Aunimod}
The normalized coefficients and absolute values of the coefficients of the distance characteristic polynomial are not unimodal (and hence not log-concave) for the Heawood graph $H$  shown in Figure \ref{figHeawPlot}.  The coefficients of the distance characteristic polynomial are log-concave but not unimodal.  

\begin{figure}[h!]
\begin{center}
\scalebox{.33}{\includegraphics{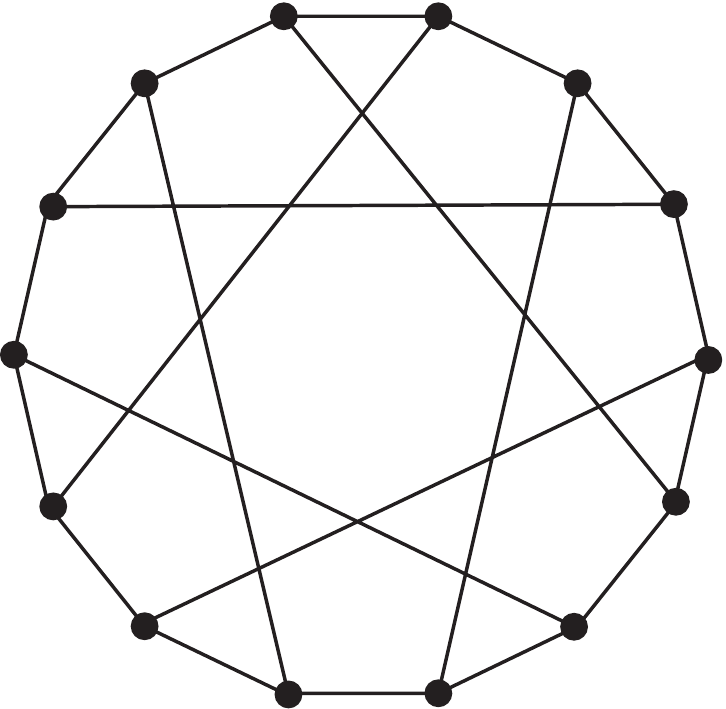}}
\caption{The Heawood graph $H$
 \label{figHeawPlot}}
\end{center}
\end{figure}

The distance characteristic polynomial of $H$ is
\bea
p_{\D(H)}(x)&= & x^{14} - 441x^{12} - 6328x^{11} - 36456x^{10} - 75936x^9 + 104720x^8\\
&& +\ 573696x^7 - 118272x^6 - 1885184x^5 + 973056x^4 \\
&& +\ 2795520x^3 -
3885056x^2 + 1892352x - 331776.
\eea
The values of $d_k(H)$, for $k=0,\dots,12$  are 
\[
81,~924,~3794,~5460,~3801,~14728,~1848,~17928,~6545,~9492,~9114,~3164,~441.
\]
\end{ex}

\bigskip
\noindent\textbf{Acknowledgment.}  We thank Ben Braun, Steve Butler, Jay Cummings, Jessica De Silva,  Wei Gao, and Kristin Heysse for stimulating discussions, and gratefully acknowledge financial support for this research from NSF 1500662, Elsevier, and the International Linear Algebra Society.  


\end{document}